\documentclass[11pt]{article}
\usepackage{latexsym,amsfonts,euscript}
\usepackage{amssymb}
\usepackage{amsthm}
\usepackage{stmaryrd}
\usepackage{mathrsfs}
\usepackage{leftidx}
\usepackage{indentfirst}
\usepackage{color}

\makeatletter 
\@addtoreset{equation}{section}
\makeatother  

\def\la{\langle}\def\ra{\rangle}

\def\pfend{\hfill{$\square$}}

\def\cent#1#2{{\bf C}_{#1}(#2)}
\def\syl#1#2{{\rm Syl}_#1(#2)}
\def\norm#1#2{{\bf N}_{#1}(#2)}
\def\oh#1#2{{\bf O}_{#1}(#2)}

\def\gfitt#1{{\bf F^*}(#1)}
\def\z#1{{\bf Z}(#1)}

\newtheorem{lem}{ \bf Lemma}[section]

\newtheorem{prop}[lem]{\bf Proposition}

\newtheorem*{rmk*}{\bf Remark}
\newtheorem*{prop*}{\bf Proposition}

\newtheorem*{thm*}{\bf Theorem}
\newtheorem*{thmA}{\bf Theorem A}

\newtheorem*{thmC}{\bf Theorem D}

\newtheorem{cor}[lem]{\bf Corollary}
\newtheorem*{cor*}{\bf Corollary}
\newtheorem*{corB}{\bf Corollary B}
\newtheorem*{corC}{\bf Corollary C}
\newtheorem{hy}[lem]{\bf Hypothesis}
\newtheorem*{hy*}{\bf Hypothesis}

\title{On $\mathcal{M}$-supplemented subgroups
\thanks{The project is supported by the NSF of China (No. 12171058)}}
\date{}

\author{Yu Zeng\\
{\footnotesize\small  Dept. Mathematics, Changshu Institute of
Technology, Changshu, Jiangsu, 215500, China}\\
{\footnotesize\small E-mail: yuzeng2004@163.com}}

\begin{document}
\maketitle
\date{}

\vskip 1cm

\begin{center}\textbf{Abstract}\end{center}

Let $G$ be a finite group and $p^k$ be a prime power dividing $|G|$.
A subgroup $H$ of $G$ is called to be \emph{$\mathcal{M}$-supplemented} in
$G$ if there exists a subgroup $K$ of $G$ such that $G=HK$ and
$H_iK<G$ for every maximal subgroup $H_i$ of $H$. In this paper, we
complete the classification of the finite groups $G$ in which all
subgroups of order $p^k$ are $\mathcal{M}$-supplemented. In
particular, we show that if $k\geq 2$, then $G/\oh{p'}{G}$ is
supersolvable with a normal Sylow $p$-subgroup and a cyclic
$p$-complement.

\vskip 5cm

\bigskip


\textbf{2020 MR Subject Classification}\,\, 20D10, 20D20
 \pagebreak

\section{Introduction}

\noindent In this paper, $G$ always denotes a finite group, $p$
always denotes a prime. Recall that a subgroup $H$ of $G$ is
\emph{complemented} in $G$ if there exists a subgroup $K$ of $G$ such that
$G=HK$ and $H\cap K=1$. A well-known theorem of Hall states that $G$
is solvable if and only if all its Sylow subgroups are complemented.
Inspired by Hall's work, a number of authors studied the structure of $G$ in which certain families of subgroups are complemented.
For instance, Monakhov and Kniahina \cite{monakhov2015}
investigated the $pd$-composition factors (the composition factors
of order divisible by $p$) of $G$ under the assumption that all
minimal $p$-subgroups are complemented in $G$, where $p$ is a given
prime divisor of $|G|$.

Following \cite{ml2009}, a subgroup $H$ of $G$ is said to be
\emph{$\mathcal{M}$-supplemented} in $G$ if $G$ admits a subgroup $K$ such
that $G=HK$ and $H_iK<G$ for every maximal subgroup $H_i$ of $H$; if
this happens, $K$ is called an \emph{$\mathcal{M}$-supplement} of $H$ in
$G$. Generally speaking, an $\mathcal{M}$-supplemented subgroup is
not necessarily a complemented subgroup. For instance, every
subgroup of order $4$ in the quaternion group of order $8$ has an
$\mathcal{M}$-supplement, however it has no complement.
Nevertheless, when $H<G$ has prime order, it is easy to see that $H$
is complemented in $G$ if and only if $H$ is
$\mathcal{M}$-supplemented in $G$.  In \cite{mz2018},  Miao and
Zhang described the nonabelian $pd$-chief factors (the chief factors
of order divisible by $p$) of $G$ under the following hypothesis.

\begin{hy}\label{hy}
Let $p^k$ be a prime power such that $p\leq p^k\leq |G|_p$, and
assume  that all subgroups of order $p^k$  of $G$ are
$\mathcal{M}$-supplemented in $G$.
    \end{hy}

For simplicity, if $G$ satisfies Hypothesis \ref{hy}, then we write
$G\in \mathcal{M}(p^k)$. In this paper, we will give a
classification (completely description) of finite groups $G$
satisfying Hypothesis \ref{hy}. Note that if $G$ satisfies the
hypothesis with $k=1$, that is, all minimal $p$-subgroups of $G$ are
complemented in $G$,  we have already obtained a classification in
\cite[Theorem 1.1]{zeng2019}. Now we only need to investigate the
case when $k\geq 2$. 
It turns out that if $G$ satisfies  Hypothesis \ref{hy} with $k\geq 2$,
then $G$ has no nonabelian $pd$-chief factor. In fact, we get the
following theorem.

\bigskip

\begin{thmA}
    Assume that $\oh{p'}{G}=1$. Then $G\in\mathcal{M}(p^k)$
    where $k\geq 2$ if and only if
    the following statements hold.

    {\rm (1)} $G=\la x\ra\ltimes P$ where $P\in\mathrm{Syl}_p(G)$;

    {\rm (2)} $\Phi(G)=\Phi(P)$ is contained in all subgroups of $G$ of order $p^{k-1}$;

    {\rm (3)} $x$ acts faithfully and via scalar multiplication on the vector space $P/\Phi(P)$.
\end{thmA}

\bigskip

Observe that $G$ satisfies Hypothesis \ref{hy} if and only if
$G/\oh{p'}{G}$ satisfies it (see Corollary 2.3). Hence
Theorem A and \cite[Theorem 1.1]{zeng2019} give a complete
description of finite groups $G$ satisfying Hypothesis \ref{hy} with $k\geq 2$.

The next two corollaries follow directly from Theorem A.
\bigskip

\begin{corB}
    Assume that $G\in
    \mathcal{M}(p^k)$ where $k\geq 2$. Then $G/\oh{p'}{G}$ is
    a supersolvable group with a normal Sylow $p$-subgroup and a cyclic $p$-complement.
\end{corB}

\bigskip

\begin{corC}
    A $p$-group $G\in\mathcal{M}(p^k)$ if and only if $\Phi(G)$ is contained in all subgroups of $G$ of order $p^{k-1}$.
\end{corC}

\bigskip

We say that $G\in \mathcal{M}(p^k)$ is of critical type
if $k\geq 2$, $\oh{p'}{G}=1$ and $|\Phi(G)|=p^{k-1}$. In order to prove Theorem
A, the first step is to study the groups of critical type.

\bigskip

\begin{thmC}
    Assume that  $G\in \mathcal{M}(p^k)$ is
    of critical type. Then one and only one  of the following holds.

    {\rm (1)}  $G=H\ltimes P$ where $H\lesssim C_{p-1}$ and
$P=\cent{G}{P}\cong C_{p^k}$;

{\rm (2)} $p^k=2^2$ and $G\cong Q_8$.
\end{thmC}

Note that the main results in \cite{mz2018}, \cite{monakhov2015} and \cite{zeng2019} depend on
the classification of finite simple groups (CFSG), but the proofs in
this paper are CFSG-free.

\section{Proofs}

\noindent We use standard notation in group theory, as in
\cite{huppert}. 
We start with  some elementary properties of
$\mathcal{M}$-supplemented subgroup.

\begin{lem} \label{l201} {\rm (\cite[Lemma 2.1]{ml2009})}
 Let $H$ be a subgroup of $G$. Then the following hold.

{\rm (1)} If $H$ is $\mathcal{M}$-supplemented in $G$, then $H$ is
$\mathcal{M}$-supplemented in $M$ whenever $H \leq M\leq G$;

{\rm (2)} If $H$ is $\mathcal{M}$-supplemented in $G$, then $H/N$ is
$\mathcal{M}$-supplemented in $G/N$ for every $G$-invariant subgroup $N$ of $H$;

{\rm (3)}  If $N$ is a normal subgroup of $G$ with $(|H|, |N|)=1$,
then $H$ is $\mathcal{M}$-supplemented in $G$ if and only if $HN/N$
is $\mathcal{M}$-supplemented in $G/N$.
\end{lem}

\begin{cor}\label{cr}
  Let $G$ be a group. Then the following statements hold.

{\rm (1)}
  If  $G\in\mathcal{M}(p^k)$, then  $H\in\mathcal{M}(p^k)$ whenever $H\leq G$ with $p^k\mid |H|$;

  {\rm (2)}
  If $G\in\mathcal{M}(p^k)$,
  then $G/N\in \mathcal{M}(p^{k-s})$ for every $G$-invariant $p$-subgroup $N$ of $G$
  such that $|N|=p^s\leq p^k$;

  {\rm (3)} $G\in\mathcal{M}(p^k)$ if and only if $G/\oh{p'}{G}\in\mathcal{M}(p^k)$.
\end{cor}
\begin{proof}
    It follows directly from Lemma \ref{l201}.
\end{proof}








A more general version of the next result was proved by Wielandt in
\cite{wielandt 1967}. For the reader's convenience, we states a
proof of Proposition \ref{outerauto}  based on \cite[Theorem
3]{cameron1972}.

\begin{prop}\label{outerauto}
Let $S$ be a nonabelian simple group with $p\mid |S|$. Assume that
$S$ is a subgroup of $\mathsf{S}_p$, the symmetric group of degree
$p$.   Then $p\nmid |\mathrm{Out}(S)|$
\end{prop}
\begin{proof}
Obviously, $p$ is odd and $S$ is a transitive subgroup of
$\mathsf{S}_p$.  Let $P\in\syl{p}{S}$. Note that $P\cong C_p$  acts
regularly on $\{1,\cdots,p\}$.   This yields by Frattini argument
that $S=P\cdot H$ where $H\in\mathrm{Hall}_{p'}(S)$ is a point
stablizer. Since $S$ is nonsolvable, it follows by \cite[Corollary
3.5B]{dixon} that $S$ is a doubly transitive group. 
Note that $\cent{\mathsf{S}_p}{S}\leq \cent{\mathsf{S}_p}{P}=P\leq S$ and $\cent{\mathsf{S}_p}{S}\cap S=1$, and so $\cent{\mathsf{S}_p}{S}=1$.
Without loss of
generality, we identify $\norm{\mathsf{S}_p}{S}$ and $S$ as
subgroups of $\mathrm{Aut}(S)$.
  Then by \cite[Theorem 3]{cameron1972}, $\norm{\mathsf{S}_p}{S}\unlhd \mathrm{Aut}(S)$, and $\mathrm{Aut}(S)/\norm{\mathsf{S}_p}{S}$ is a $2$-group.
  So it suffices to show $p\nmid |\norm{\mathsf{S}_p}{S}/S|$.
  Write $N=\norm{\mathsf{S}_p}{S}$.
  Then by Frattini argument, $N=S\norm{N}{P}$.
  Hence
  \begin{center}
    $|N/S|=|\norm{N}{P}/\norm{S}{P}|~\Big{|}~|\norm{N}{P}:\cent{S}{P}|$.
  \end{center}
Observe that $\cent{S}{P}\leq \cent{\mathsf{S}_p}{P}=P$, and so  
\[
  \norm{N}{P}/\cent{S}{P}=\norm{N}{P}/P\leq \norm{\mathsf{S}_p}{P}/P= \norm{\mathsf{S}_p}{P}/\cent{\mathsf{S}_p}{P}\lesssim {\rm Aut}(C_p)\cong C_{p-1},
  \]
  and the result follows.
\end{proof}

\begin{lem}\label{frat} {\rm (\cite[Lemma 2.4]{GZ})} Let $P$ be a normal Sylow $p$-subgroup of
$G$. Then  $\Phi(P)=P\cap \Phi(G)$.\end{lem}

 \begin{lem} \label{mininorm}{\rm(\cite[Theorem 3.4]{mz2018})}
 Let $E$ be a minimal normal subgroup of $G$ such that $p\mid |E|$ and
$E\cap \Phi(G)=1$. If $G\in\mathcal{M}(p^k)$, then $E\lesssim \mathsf{S}_p$.
\end{lem}

Although the main result of \cite{mz2018} depends on CSFG, the proof of above lemma is CSFG-free.
In the sequel, we use $\mathrm{Sub}(G|p^k)$ to denote the set of subgroups of order $p^k$ of $G$.

\begin{prop}\label{pk-1}
Assume that $G\in\mathcal{M}(p^k)$ is a $p$-group.
Then

    {\rm (1)} If $|G|>p^k$, then $G'=\Phi(G)$;

    {\rm (2)} $\Phi(G)\leq \bigcap_{U\in \mathrm{Sub}(G|p^{k-1})} U$.
 \end{prop}
 \begin{proof} Note that if $k=1$, then all minimal subgroups of $G$ are complemented in $G$.
 This implies by \cite[Lemma 3.3]{zeng2019} that $G$ is elementary abelian,
 now the required results hold obviously.
    Hence we may always assume $k\geq 2$.

 (1) We work by induction on $|G|+k$.
 Suppose that $G'>1$, and let $N$ be a minimal $G$-invariant subgroup of $G'$.
    Since $G/N\in \mathcal{M}(p^{k-1})$ by part (2) of
  Corollary \ref{cr}, we conclude by induction that $(G/N)'=\Phi(G/N)$.
  This implies $G'=\Phi(G)$, and we are done.

 Suppose that $G'=1$.
 Assume that $\Phi(G)>1$, and let $N$ be a minimal $G$-invariant subgroup of $\Phi(G)$.
 We also conclude by induction that $\Phi(G/N)=(G/N)'=1$.
 This implies that $N=\Phi(G)$ for all minimal normal subgroups $N$ of $G$,
 and hence $\Phi(G)$ is
 the unique minimal normal subgroup of $G$. Since $G$ is abelian,
 $G$ is necessarily cyclic.
 Observe that $G/\Phi(G)$ is elementary
 abelian, and it follows that $G\cong C_{p^2}$. Whence $k=1$, a contradiction.
 Thus $\Phi(G)=1$, and we are done.

(2) Let $K\in \mathrm{Sub}(G|p^{k})$. We claim that $\Phi(G)\leq K$.
And we proceed by induction on $|G|+k$.
Obviously we may assume $|\Phi(G)|\geq p$ and $|G|\geq p^{k+2}$.

 Suppose that $K_G:=\bigcap_{g\in G} K^g=1$.
Let $U$ be a maximal subgroup of $K$ and let $Z$ be a minimal normal
subgroup of $G$. Since $G/Z\in \mathcal{M}(p^{k-1})$ by part (2) of Corollary \ref{cr} and
$|UZ/Z|=p^{k-1}$, we have $\Phi(G)Z/Z\leq UZ/Z$ by induction. Hence
     \[
      \Phi(G)\leq \bigcap_U UZ,
     \]
     where $U$ runs over all maximal subgroups of $K$.
     Note that $K\cap Z=1$, it is easy to see that
     $$\Phi(G)\leq \bigcap_U UZ=(\bigcap_U U)Z=Z.$$
By the arbitrariness of $Z$, $\Phi(G)$ is the unique minimal normal
subgroup of
     $G$. In particular $\z{G}$ is cyclic.
Let $A$ be a maximal subgroup of $G$ with $K\leq A$. Since
$A\in\mathcal{M}(p^k)$ by part (1) of Corollary \ref{cr}, it follows by
induction that $\Phi(A)\leq K$. As $K_G=1$ and $\Phi(A)\unlhd G$,
$A$ is elementary abelian. Since $\Phi(G)=G'$ by (1), $G$ is
nonabelian. 
Observe that $G$ is abelian if $\z{G}\nleq A$, and so
$\z{G}\leq A$. Recall that $A$ is  elementary abelian  and $\z{G}$
is cyclic, it follows that $\z{G}=\Phi(G)=G'$ has order $p$. Now $G$
is an extraspecial $p$-group admitting an abelian maximal subgroup.
     So $|G|=p^3$, which contradicts the assumption that $|G|\geq p^{k+2}\geq p^4$.

Therefore $K_G>1$. Let $Z$ be a minimal $G$-invariant subgroup of
$K$. Since $G/Z\in \mathcal{M}(p^{k-1})$ by part (2) of Corollary \ref{cr} and $K/Z\in
\mathrm{Sub}(G/Z|p^{k-1})$,
 we conclude by induction that  $\Phi(G/Z)\leq K/Z$. Hence $\Phi(G)\leq K$, and the claim follows.

Let $U\in \mathrm{Sub}(G|p^{k-1})$ and let $D\in
\mathrm{Sub}(G|p^k)$ be  such that $U< D$.
     Since $G\in \mathcal{M}(p^k)$, there is an $\mathcal{M}$-supplement $B$ of $D$ in $G$.
 Since $UB<G$ by the definition of $\mathcal{M}$-supplemented subgroup, $UB$ is maximal in $G$.
     This implies $\Phi(G)\leq UB$.
 Since $\Phi(G)\leq D$ by the claim,
 we have $\Phi(G)\leq D\cap UB=U(D\cap B)=U(\Phi(D)\cap B)=U$, as wanted.
 \end{proof}

 \bigskip

\noindent{\em Proof of Theorem D.}~~We only need to show that if
$G\in \mathcal{M}(p^k)$ is of critical type, then $G$ is one of the
groups listed in Theorem D. Suppose that $G\in \mathcal{M}(p^k)$ is
of critical type. Then $\oh{p'}{G}=1$, $|\Phi(G)|=p^{k-1}$, $k\geq
2$. Let $P\in {\rm Syl}_p(G)$.

Let $D\in \mathrm{Sub}(P|p^k)$, and assume that $D$ is not cyclic.
Since $\Phi(G)<D$ by Proposition \ref{pk-1}, there exists a maximal
subgroup $U$ of $D$  such that $D=U\Phi(G)$. Let $B$ be an
$\mathcal{M}$-supplement of $D$ in $G$. We have
$G=DB=(UB)\Phi(G)=UB< G$, a contradiction. 
Hence all subgroups of $P$ of order $p^k$ are cyclic.
It follows by \cite[Chapter III, Theorem 8.4]{huppert} that $P$ is either a cyclic group or a
            generalized quaternion group.

Let $L/\Phi(G)$ be the socle of $G/\Phi(G)$, that is, the product of
all minimal normal subgroups of $G/\Phi(G)$. Since
$\Phi(G/\Phi(G))=1$, $L/\Phi(G)=\gfitt{G/\Phi(G)}$. So $G/L\lesssim
{\rm Out}(L/\Phi(G))$.

Case 1. Assume that $P$ is cyclic.

Since $P\in\mathcal{M}(p^k)$  by  Corollary \ref{cr}, it follows by
Proposition \ref{pk-1} that  $|P|=p^k$ and  $|P/\Phi(G)|=p$.
Obviously $\oh{p'}{G/\Phi(G)}=1$ because  $\oh{p'}{G}=1$. This
implies that  $L/\Phi(G)$ is minimal normal in $G/\Phi(G)$ with
$|L/\Phi(G)|_p=p$.

 Suppose that $L/\Phi(G)$ is nonabelian. Since
$G/\Phi(G)\in\mathcal{M}(p)$ by Corollary \ref{cr}, it follows by
Lemma \ref{mininorm} that $L/\Phi(G)\lesssim \mathsf{S}_p$. Since
$L/\Phi(G)$ has a cyclic Sylow $p$-subgroup, it follows from
\cite[Chapter V, Theorem 25.3]{huppert} that $p$ does not divide the
order of the Schur multiplier $\mathrm{M}(L/\Phi(G))$ of
$L/\Phi(G)$. However, the cyclicity of $P$ shows $C_L(\Phi(G))=L$
and $L=L'$, this leads to $\Phi(G)\lesssim {\rm M}(L/\Phi(G))$, a
contradiction. Therefore $L/\Phi(G)$ is abelian, whence $G/L\lesssim
C_{p-1}$ and (1) holds.

Case 2. Assume that $P$ is a generalized quaternion group.

            Clearly, $p=2$ and $P/\Phi(P)\cong C_2\times C_2$.
            Since all subgroups of order $2^k$ in $P$ are cyclic ($2^k\geq 4$),
            we have $k=2$ and so $\Phi(G)\cong C_2\cong \z{P}$.
Observe that  $P/\Phi(G)\in\mathcal{M}(2)$ by Corollary \ref{cr},
that is, every subgroup of $P/\Phi(G)$ of order 2 has a complement
in $P/\Phi(G)$. It follows that $P/\Phi(G)$ is elementary abelian.
Now $\Phi(G)=\Phi(P)\cong C_2$ and $P\cong Q_8$. Note that
$G/\Phi(G)\in \mathcal{M}(2)$. Applying Lemma \ref{mininorm} to
$G/\Phi(G)$, we conclude that $L/\Phi(G)$ is a direct product of two
minimal subgroups $N_1$ and $N_2$ of $G/\Phi(G)$ where $N_1\cong
N_2\cong C_2$. This implies $G/L=1$, and  (2) holds. \pfend

\bigskip

\begin{lem}\label{mainlem}
Let $G\in\mathcal{M}(p^k)$, where $k\geq 2$, be such that $\oh{p'}{G}=\Phi(G)=1$, and let $P$ be the socle of $G$.
Then $G=H\ltimes P$,
where $P\in\syl{p}{G}$ is elementary abelian and $H=\la x\ra$,
and there is a positive integer $d$ with $d\equiv |H|({\rm mod} ~p)$ such that $v^x=v^d$ for all $v\in P$.
\end{lem}
\begin{proof}
    Since $\oh{p'}{G}=\Phi(G)=1$, we have by Lemma \ref{mininorm} that
        $$P=\gfitt{G}=\la x_1\ra \times\cdots\times \la x_a\ra \times S_1\times \cdots \times S_b,$$
    where $o(x_1)=\cdots =o(x_a)=p$, and all $S_j$ are nonabelian simple groups
    such that $S_j\lesssim \mathsf{S}_p$ and $|S_j|_p=p$.
    Since $\cent{G}{\gfitt{G}}\leq \gfitt{G}$, we have
        $$G/P\lesssim {\rm Out}(\la x_1\ra)\times \cdots \times {\rm Out}(\la x_a\ra)\times {\rm Out}(S_1)\times \cdots \times {\rm Out}(S_b).$$
Note that all ${\rm Out}(\la x_i\ra)$ and ${\rm Out}(S_j)$ are
$p'$-groups by Proposition \ref{outerauto}. Consequently $G/P$ is a
$p'$-group, and $P_0\in {\rm Syl}_p(G)$ is elementary abelian. It
follows that if $B$ is an $\mathcal{M}$-supplement of
    a member $D\in \mathrm{Sub}(G|p^k)$,  then $B$ is exactly a complement of $D$ in $G$.

We claim that $b=0$, that is,  $P$ is elementary abelian. Assume the
claim is not true and let $G$ be a counterexample with smallest
possible sum $|G|+k$. By the minimality of $|G|+k$  and parts (1),
(2) of Corollary \ref{cr},
    $$G=\la x\ra \times S,$$
    where $o(x)=p$,  $S\lesssim \mathsf{S}_p$ is a nonabelian simple group with $|S|_p=p$,
    and $G\in\mathcal{M}(p^2)$.
    Let $y\in S$ be of order $p$, and let $D=\la x\ra \times \la y\ra$,
    also let $B$ be an $\mathcal{M}$-supplement of $D$ in $G$. We
    have
       \begin{center}
       $G=DB$ with $D\cap B=1$, and $UB<G$ for all $U\in \mathrm{Sub}(D|p)$.
    \end{center}
Note that $\la xy\ra B<G$ and that $G=DB=(\la x\ra \times \la xy\ra)
B=\la x\ra (\la xy\ra B)=\la x\ra \times (\la xy\ra B)$. Hence $\la
xy\ra B\unlhd G$ and $\la xy\ra B\cong S$. Since $B$ is the unique
nonabelian minimal normal subgroup of $G$, we have $\la xy\ra B=S$.
This leads to $x\in S$, a contradiction.

Therefore  $P$ is elementary abelian, and hence $G=H\ltimes P$,
where $H\in {\rm Hall}_{p'}(G)$ and $P=P_0\in\syl{p}{G}$.
    We claim that $U\unlhd G$ for every $U\in  \mathrm{Sub}(P|p^{k-1})$.
    Suppose the claim is false and let $G$ be a minimal counterexample.
    Then $U^h\neq U$ for some $h\in H$.
    Let $D\in \mathrm{Sub}(P|p^k)$ be such that $U<D$,
    and let $B$ be an $\mathcal{M}$-supplement of $D$ in $G$.
    Note that $B$ is exactly a complement of $D$ in $G$. We have
    that
    $$G=DB\,\,{with}\,\, D\cap B=1, \,\, UB<G\,\,{with}\,\,|G:UB|=p.$$
    Since $G$ is $p$-solvable, we may assume that $H\leq B$.
    Notice that $UB=H\ltimes (P\cap UB)$  admits a normal abelian Sylow $p$-subgroup,
    and so $|UB|_p\geq |UU^h|\geq p|U|=p^{k}$.
    By part (1) of Corollary \ref{cr}, $UB\in\mathcal{M}(p^k)$.
By the minimality of $G$, we have $U\unlhd UB$,
    and thus $U\unlhd G$, a contradiction.
    Hence every $U\in \mathrm{Sub}(P|p^{k-1})$ is normal in $G$.
    Since $P$ is elementary abelian,
    this also implies that all subgroups of order $p$ of $P$ are normal in $G$.
Now  $P$ is a faithful homogenous $\mathbb{F}_p[H]$-module with
every irreducible component having dimension 1.
    Thus $H=\la x\ra$ is cyclic.
    Since all subgroups of $P$ of order $p$ are $\la x\ra$-invariant,
    there exists a positive integer $d$ with $d\equiv |H|({\rm mod} ~p)$ such that $v^x=v^d$ for all $v\in P$.  
\end{proof}

\bigskip

\noindent\emph{Proof of Theorem A.}~~Note that $\Phi(G)$ is a $p$-group. We write $|\Phi(G)|=p^s$ and $\overline{G}=G/\Phi(G)$.

($\Rightarrow$) Since every subgroup of $\Phi(G)$ is not $\mathcal{M}$-supplemented in $G$,
it follows from $G\in\mathcal{M}(p^k)$ that $k-s\geq 1$. 
Suppose that $k-s=1$. Then $G\in \mathcal{M}(p^k)$ is of critical
type, and Theorem D implies the required result. We then assume that
$k-s\geq 2$. Clearly, $\Phi(\overline{G})=\oh{p'}{\overline{G}}=1$,
and $\overline{G}\in\mathcal{M}(p^{k-s})$ by Corollary \ref{cr}.  By
Lemma \ref{mainlem}, we have $\overline{G}=\la
\overline{x}\ra\ltimes \overline{P}$, where
$\cent{\overline{G}}{\overline{P}}=\overline{P}$,  and
there exists a positive integer $d$ with $d\equiv
o(\overline{x})(\mathrm{mod}~p)$ such that $\overline{v}^{\overline{x}}=\overline{v}^d$ for all
$\overline{v}\in \overline{P}$.
Clearly, $\cent{\overline{G}}{\overline{P}}=\overline{P}$ implies
$\cent{G}{P}\leq P$.
Observing that the Sylow $p$-subgroup $P$ of $G$
is normal in $G$, we have $\Phi(G)=P\cap \Phi(G)=\Phi(P)$ by Lemma
\ref{frat}.  We may assume $x$ is a $p'$-element of $G$, and then
$G=\la x\ra\ltimes P$. 
Clearly, $x$ acts faithfully and via scalar multiplication on the vector
space $P/\Phi(P)$.
Finally, $\Phi(P)\leq \bigcap_{U\in\mathrm{Sub}(G|p^{k-1})} U$ by part (2) of Proposition \ref{pk-1}.

($\Leftarrow$) Let $D\in \mathrm{Sub}(G|p^k)$ and  let $U$ be a
maximal subgroup of $D$. By Condition (2), $\Phi(G)=\Phi(P)\leq
U\leq D$. Since every $p$-subgroup of $\overline{G}$ is normal in
$\overline{G}$, we have $\overline{P}=\overline{D}\times
\overline{C}$ where $C\unlhd G$ with $\Phi(P)\leq C$. Write $B=\la
x\ra C$. It is routine to check that $G=DB$ and $UB<G$. Hence $D$ is
$\mathcal{M}$-supplemented in $G$, and $G\in \mathcal{M}(p^k)$.
\pfend

\bigskip

\noindent\emph{Proof of Corollary B.}~~By Corollary 2.3, we may assume by induction that $\oh{p'}{G}=1$. Then $G$ is a group in Theorem A.
    Observe that $x$ acts faithfully and via scalar multiplication on the vector space $P/\Phi(G)$.
    This implies that all $G$-chief factors of $P/\Phi(G)$ have order $p$. Consequently, $G/\Phi(G)$ is
    supersolvable, and so is $G$.
\pfend

\bigskip

\noindent\emph{Proof of Corollary C.}~~If $k\geq 2$, then Theorem A yields the result. If $k=1$, then
$G\in\mathcal{M}(p)$, that is every subgroup of order $p$ of
$G$ is complemented in $G$, which is equivalent to $G$ being
elementary abelian by \cite[Lemma 3.3]{zeng2019}.
\pfend

\bigskip




    {\bf Acknowledgement:} The author gratefully acknowledges
    the support of China Scholarship Council (CSC),
    and the author wish to thank Professor Dolfi and Professor Qian for their valuable comments.

\end{document}